\documentclass[12pt]{amsart}

\usepackage{accents}
\usepackage{appendix}
\usepackage{amsfonts}
\usepackage{amsmath}
\usepackage{amssymb}	
\usepackage{amsthm,bm}
\usepackage{array,booktabs,multirow}
\usepackage{braket}
\usepackage{centernot}
\usepackage{cite}
\usepackage{comment}
\usepackage{dsfont}
\usepackage{mathalfa}
\usepackage[shortlabels]{enumitem} 
\usepackage{etoolbox}
\usepackage{float}
\usepackage[hang, flushmargin]{footmisc}
\usepackage{latexsym}
\usepackage{lipsum}
\usepackage{needspace}
\usepackage{tikz}
\usepackage{hyperref}
\usetikzlibrary{matrix,arrows}  
\usepackage{kbordermatrix}

\makeatletter
\def\namedlabel#1#2{\begingroup
   \def\@currentlabel{#2}%
   \label{#1}\endgroup
}
\makeatother

\renewcommand{\kbldelim}{(}
\renewcommand{\kbrdelim}{)}

\theoremstyle{plain}
\newtheorem{thm}{Theorem}[section]

\newtheorem{lem}[thm]{Lemma}
\newtheorem{prop}[thm]{Proposition}

\theoremstyle{definition}

\newtheorem{defn}[thm]{Definition}

\theoremstyle{remark}

\setlist[enumerate,1]{leftmargin=2em}

\def\H{\mathfrak H}

\def\F{\mathbb F}

\def\e{\varepsilon}

\title[The additive DAHA and the Leonard triples]{The universal additive DAHA of type $(C_1^\vee,C_1)$ and Leonard triples}

\author{Si-Yao Huang}
\address{
Si-Yao Huang\\
Department of Mathematics\\
National Central University\\
Chung-Li 32001 Taiwan
}

\author{Hau-Wen Huang}
\address{
Hau-Wen Huang\\
Department of Mathematics\\
National Central University\\
Chung-Li 32001 Taiwan
}
\email{hauwenh@math.ncu.edu.tw}

\begin{document}
\begin{abstract}
Assume that $\F$ is an algebraically closed field with characteristic zero. The universal Racah algebra $\Re$ is a unital associative $\F$-algebra generated by $A,B,C,D$ and the relations state that 
$[A,B]=[B,C]=[C,A]=2D$ and each of 
$$
[A,D]+AC-BA,
\qquad
[B,D]+BA-CB,
\qquad
[C,D]+CB-AC
$$
is central in $\Re$. The universal additive DAHA (double affine Hecke algebra) $\H$ of type $(C_1^\vee,C_1)$ is a unital associative $\F$-algebra generated by $\{t_i\}_{i=0}^3$ and the relations state that 
\begin{gather*}
t_0+t_1+t_2+t_3 = -1,
\\
\hbox{$t_i^2$ is central for all $i=0,1,2,3$}.
\end{gather*}
Any $\H$-module can be considered as a $\Re$-module via the $\F$-algebra homomorphism $\Re\to \H$ given by
\begin{eqnarray*}
A &\mapsto & \frac{(t_0+t_1-1)(t_0+t_1+1)}{4},
\\
B &\mapsto & \frac{(t_0+t_2-1)(t_0+t_2+1)}{4},
\\
C &\mapsto & \frac{(t_0+t_3-1)(t_0+t_3+1)}{4}.
\end{eqnarray*}
Let $V$ denote a finite-dimensional irreducible $\H$-module. In this paper we show that $A,B,C$ are diagonalizable on $V$ if and only if $A,B,C$ act as Leonard triples on all composition factors of the $\Re$-module $V$. 
\end{abstract}

\thanks{The purpose of this unpublished work is only to make sure the existence of $q\to 1$ version of \cite{Huang:DAHA&LT} and this work is based on the master thesis by the first author under the supervision of the second author}

\maketitle



\allowdisplaybreaks

\section{Introduction}\label{s:introduction}

Throughout this paper, we adopt the following conventions: Assume that the ground field $\F$ is algebraically closed with characteristic zero. An algebra is meant to be a unital associative algebra over $\F$. An algebra homomorphism is meant to be a unital algebra homomorphism.

\begin{defn}
[Definition 2.1, \cite{SH:2019-1}]
\label{defn:UAW}
The  {\it universal Racah algebra} $\Re$ is an algebra defined by generators and relations in the following way: The generators are $A, B, C, D$ and the relations state that
\begin{gather}\label{central}
[A,B]=[B,C]=[C,A]=2D
\end{gather}
and each of
\begin{gather*}\label{UAW_central}
[A,D]+AC-BA,
\qquad
[B,D]+BA-CB,
\qquad
 [C,D]+CB-AC
\end{gather*}
commutes with $A, B, C, D$.

Define
\begin{gather}\label{UAW_central}
\delta=A+B+C.
\end{gather}
Using (\ref{central}) yields that $\delta$ is central in $\Re$.
\end{defn}

\begin{defn}
[\!\!\cite{BI&NW2016, Groenevelt2007}] 
\label{defn:H}
The {\it universal additive DAHA (double affine Hecke algebra) $\H$ of type $(C_1^\vee,C_1)$} is an algebra defined by generators and relations. The generators are $t_0,t_1,t_2,t_3$ and the relations state that 
\begin{gather}\label{t0123}
t_0+t_1+t_2+t_3=-1
\end{gather}
and each of $t_0^{2},t_1^{2},t_2^{2},t_3^{2}$ commutes with $t_0,t_1,t_2,t_3$.
\end{defn}

\begin{thm}
[\!\!\cite{R&BI2015, Huang:R<BI}]
\label{thm:hom}
There exists a unique algebra homomorphism $\zeta:\Re\to \H$ that sends 
\begin{eqnarray*}
A &\mapsto &\frac{(t_0+t_1-1)(t_0+t_1+1)}{4},
\\
B &\mapsto & \frac{(t_0+t_2-1)(t_0+t_2+1)}{4},
\\
C &\mapsto & \frac{(t_0+t_3-1)(t_0+t_3+1)}{4},
\\
\delta &\mapsto & \frac{t_0^{2}+t_1^{2}+t_2^{2}+t_3^{2}}{4}-\frac{t_0}{2}-\frac{3}{4}.
\end{eqnarray*}
\end{thm}

Note that $\zeta$ is the Racah version of the algebra homomorphism from the universal Askey--Wilson algebra $\triangle_q$ into the universal DAHA $\H_q$. Let $V$ denote a finite-dimensional irreducible $\H_q$-module. In  \cite{Huang:DAHA&LT,daha&LP} it was shown that the defining generators of $\triangle_q$ are diagonalizable on $V$ if and only if the defining generators of $\triangle_q$ act as Leonard triples on all composition factors of the $\triangle_q$-module $V$.

In this paper we are devoted to giving the Racah version of \cite{Huang:DAHA&LT,daha&LP}.
We begin with recalling the notion of Leonard pairs and Leonard triples. We will use the following terms: A square matrix is said to be {\it tridiagonal} if each nonzero entry lies on either the diagonal, the superdiagonal, or the subdiagonal. A tridiagonal matrix is said to be {\it irreducible} if each entry on the superdiagonal is nonzero and each entry on the subdiagonal is nonzero. Let $V$ denote a vector space over $\F$ with finite positive dimension. An eigenvalue $\theta$ of a linear operator $L:V\to V$ is said to be {\it multiplicity-free} if the algebraic multiplicity of $\theta$ is equal to $1$. A linear operator $L:V\to V$ is said to be {\it  multiplicity-free} if all eigenvalues of $L$ are multiplicity-free.

\begin{defn}
[Definition 1.1, \cite{lp2001}]
\label{defn:lp}
Let $V$ denote a vector space over $\F$ with finite positive dimension. By a {\it Leonard pair} on $V$, we mean a pair of linear operators $L: V \to V$ and $L^* : V \to V$ that satisfy both (i), (ii) below.
\begin{enumerate}
\item There exists a basis for $V$ with respect to which the matrix representing $L$ is diagonal and the matrix representing $L^*$ is irreducible tridiagonal. 

\item There exists a basis for $V$ with respect to which the matrix representing $L^*$  is diagonal and the matrix representing $L$ is irreducible tridiagonal.
\end{enumerate}
\end{defn}

\begin{defn}
[Definition 1.2, \cite{cur2007}]
\label{defn:lt}
Let $V$ denote a vector space over $\F$ with finite positive dimension. 
By a {\it Leonard triple} on $V$, we mean a triple of linear operators $L:V\to V,L^* :V\to V,L^\e:V\to V$ that satisfy the conditions (i)--(iii) below.
\begin{enumerate}
\item There exists a basis for $V$ with respect to which the matrix representing $L$ is diagonal and the matrices representing $L^* $ and $L^\e$ are irreducible tridiagonal.

\item There exists a basis for $V$ with respect to which the matrix representing $L^* $ is diagonal and the matrices representing $L^\e$ and $L$ are irreducible tridiagonal.

\item There exists a basis for $V$ with respect to which the matrix representing $L^\e$ is diagonal and the matrices representing $L$ and $L^* $ are irreducible tridiagonal.
\end{enumerate}
\end{defn}

The main results of this paper are as follows:

\begin{thm}\label{thm:diagonal}
Suppose that $V$ is a finite-dimensional irreducible $\H$-module. Then the following are equivalent:
\begin{enumerate}
\item $A$ {\rm (}resp. $B${\rm )} {\rm (}resp. $C${\rm )}
is diagonalizable on $V$.

\item $A$ {\rm (}resp. $B${\rm )} {\rm (}resp. $C${\rm )} is diagonalizable on all composition factors of the $\Re$-module $V$.

\item $A$ {\rm (}resp. $B${\rm )} {\rm (}resp. $C${\rm )} is multiplicity-free on all composition factors of the $\Re$-module $V$.
\end{enumerate}
\end{thm}

\begin{thm}\label{thm:lp}
Suppose that $V$ is a finite-dimensional irreducible $\H$-module. Then the following are equivalent:
\begin{enumerate}
\item $A,B$ {\rm (}resp. $B,C${\rm )} {\rm (}resp. $C,A${\rm )}
are diagonalizable on $V$.

\item $A,B$ {\rm (}resp. $B,C${\rm )} {\rm (}resp. $C,A${\rm )} are diagonalizable on all composition factors of the $\Re$-module $V$.

\item $A,B$ {\rm (}resp. $B,C${\rm )} {\rm (}resp. $C,A${\rm )} are multiplicity-free on all composition factors of the $\Re$-module $V$.

\item $A,B$ {\rm (}resp. $B,C${\rm )} {\rm (}resp. $C,A${\rm )} act as Leonard pairs on all composition factors of the $\Re$-module $V$.
\end{enumerate}
\end{thm}

\begin{thm}\label{thm:lt}
Suppose that $V$ is a finite-dimensional irreducible $\H$-module. Then the following are equivalent:
\begin{enumerate}
\item $A,B,C$
are diagonalizable on $V$.

\item $A,B,C$ are diagonalizable on all composition factors of the $\Re$-module $V$.

\item $A,B,C$ are multiplicity-free on all composition factors of the $\Re$-module $V$.

\item $A,B,C$ act as Leonard triples on all composition factors of the $\Re$-module $V$.
\end{enumerate}
\end{thm}

The paper is organized as follows. In \S\ref{s:AWmodule} we recall some preliminaries on the finite-dimensional irreducible $\Re$-modules. In \S\ref{s:Rmodule} we give some sufficient and necessary conditions for $A,B,C$ acting as a Leonard triple on finite-dimensional irreducible $\Re$-modules. In \S\ref{Pre:Hmodule} we recall some facts concerning the even-dimensional irreducible $\H$-modules. In \S\ref{s:even} we prove Theorem \ref{thm:diagonal} in the even-dimensional case. In \S\ref{Pre:Hmodule2} we recall some facts concerning the odd-dimensional irreducible $\H$-modules. In \S\ref{s:odd} we prove Theorem \ref{thm:diagonal} in the odd-dimensional case and we end this paper with the proofs for Theorems \ref{thm:lp} and \ref{thm:lt}.

\section{Preliminaries on the finite-dimensional irreducible $\Re$-modules}\label{s:AWmodule}

\begin{prop}
[Proposition 2.4, \cite{SH:2019-1}]
\label{prop:UAWd}
For any scalars $a,b,c\in \F$ and any integer $d\geq 0$, there exists a $(d+1)$-dimensional $\Re$-module $R_d(a,b,c)$  satisfying the following conditions:
\begin{enumerate}
\item There exists a basis for $R_d(a,b,c)$ with respect to which the matrices representing $A$ and $B$ are
$$
\begin{pmatrix}
\theta_0 & & &  &{\bf 0}
\\ 
1 &\theta_1 
\\
&1 &\theta_2
 \\
& &\ddots &\ddots
 \\
{\bf 0} & & &1 &\theta_d
\end{pmatrix},
\qquad 
\begin{pmatrix}
\theta_0^* &\varphi_1 &  & &{\bf 0}
\\ 
 &\theta_1^* &\varphi_2
\\
 &  &\theta_2^* &\ddots
 \\
 & & &\ddots &\varphi_d
 \\
{\bf 0}  & & & &\theta_d^*
\end{pmatrix}
$$
respectively, where 
\begin{align*}
\theta_i 
&= \textstyle
\left(
a+\frac{d}{2}-i
\right)
\left(
a+\frac{d}{2}-i+1
\right)
\qquad 
\hbox{for $i=0,1,\ldots,d$},
\\
\theta^*_i 
&= \textstyle
\left(
b+\frac{d}{2}-i
\right)
\left(
b+\frac{d}{2}-i+1
\right)
\qquad 
\hbox{for $i=0,1,\ldots,d$},
\\
\varphi_i 
&= \textstyle
i(i-d-1)
\left(
a+b+c+\frac{d}{2}-i+2
\right)
\left(
a+b-c+\frac{d}{2}-i+1
\right)
\qquad 
\hbox{for $i=1,2,\ldots,d$}.
\end{align*}

\item The element $\delta$ acts on $R_d(a,b,c)$ as scalar multiplication by
\begin{align*}
\textstyle
\frac{d}{2}
\left(
\frac{d}{2}+1
\right)
+a(a+1)+b(b+1)+c(c+1).
\end{align*}
\end{enumerate} 
\end{prop}

\begin{thm}
[Theorem 4.5, \cite{SH:2019-1}]
\label{thm:irr_UAW} 
For any scalars $a,b,c\in \F$ and any integer $d\geq 0$, the $\Re$-module $R_d(a,b,c)$ is irreducible if and only if $$
a+b+c+1, 
-a+b+c, 
a-b+c, 
a+b-c
\not\in\left\{\frac{d}{2}-i\,\bigg|\, i=1,2,\ldots,d\right\}.
$$
\end{thm}

\begin{thm}
[Theorem 6.3, \cite{SH:2019-1}]
\label{thm:onto_UAW}
Let $d\geq 0$ denote an integer. Suppose that $V$ is a $(d+1)$-dimensional irreducible $\Re$-module. Then there exist $a,b,c\in \F$ such that the $\Re$-module $R_d(a,b,c)$ is isomorphic to $V$.
\end{thm}

\begin{lem}
[Theorem 6.6, \cite{SH:2019-1}]
\label{lem:dia_UAW}
Let $d\geq 0$ denote an integer and let $a,b,c$ denote scalars in $\F$. Suppose that the $\Re$-module $R_d(a,b,c)$ is irreducible. Then the following are equivalent:
\begin{enumerate}
\item $A$ {\rm (}resp. $B${\rm )}  {\rm (}resp. $C${\rm )} is diagonalizable on $R_d(a,b,c)$.

\item $A$ {\rm (}resp. $B${\rm )}  {\rm (}resp. $C${\rm )} is multiplicity-free on $R_d(a,b,c)$.

\item $2a$  {\rm (}resp. $2b${\rm )}  {\rm (}resp. $2c${\rm )} is not in $\{i-d-1\,|\, i=1,2,\ldots,2d-1\}$.
\end{enumerate}
\end{lem}

\section{The conditions for $A,B,C$ as a Leonard triple on finite-dimensional irreducible $\Re$-modules}\label{s:Rmodule}

Define the following elements of $\Re$:
\begin{align}
\alpha 
&= [A,D]+AC-BA, \label{alpha}
\\
\beta
&= [B,D]+BA-CB, \label{beta}
\\
\gamma 
&= [C,D]+CB-AC. \label{gamma}
\end{align}

\begin{lem}
\label{presentation2}
The following relations hold in $\Re$:
\begin{align}
A^2B-2ABA+BA^2-2AB-2BA=2A^2-2A\delta+2\alpha,\label{AB}
\\
AB^2-2BAB+B^2A-2AB-2BA=2B^2-2B\delta-2\beta,\label{BA}
\\
A^2C-2ACA+CA^2-2AC-2CA=2A^2-2A\delta-2\alpha,\label{AC}
\\
C^2A-2CAC+AC^2-2AC-2CA=2C^2-2C\delta+2\gamma,\label{CA}
\\
B^2C-2BCB+CB^2-2BC-2CB=2B^2-2B\delta+2\beta,\label{BC}
\\
C^2B-2CBC+BC^2-2BC-2CB=2C^2-2C\delta-2\gamma.\label{CB}
\end{align}
\end{lem}
\begin{proof}
From (\ref{central}) and (\ref{UAW_central}) we see that $D=\frac{1}{2}[A,B]$ and $C=\delta-A-B$. The relations (\ref{AB}) and (\ref{BA}) can be obtained by using these two facts to eliminate $C,D$ from (\ref{alpha}) and (\ref{beta}) respectively. The relations (\ref{AC}) and (\ref{CA}) can be obtained by substituting $B=\delta-A-C$ and $D=\frac{1}{2}[C,A]$ into (\ref{alpha}) and (\ref{gamma}) respectively. The relations (\ref{BC}) and (\ref{CB}) can be obtained by substituting $A=\delta-B-C$ and $D=\frac{1}{2}[B,C]$ into (\ref{beta}) and (\ref{gamma}) respectively.
\end{proof}

Recall that the Pochhammer symbol is defined by
\begin{align}
(x)_n= \prod_{i=1}^{n}(x+i-1) \label{Poch}
\end{align}
for all $x\in\F$ and all integers $n\geq0$. Note that (\ref{Poch}) is taken to be $1$ when $n=0$.

\begin{lem}
\label{eigenvalues_C}
Let $d\geq0$ denote an integer and let $a,b,c\in\F$. Then there exists a basis for $R_d(a,b,c)$ with respect to which the matrix representing $C$ is
$$
\begin{pmatrix}
\theta^{\e}_0 & & &  &{\bf 0}
\\ 
1 &\theta^{\e}_1 
\\
&1 &\theta^{\e}_2
 \\
& &\ddots &\ddots
 \\
{\bf 0} & & &1 &\theta^{\e}_d
\end{pmatrix},
$$
where $\theta^{\e}_i = \textstyle
(c+\frac{d}{2}-i)
(c+\frac{d}{2}-i+1)
$ for $i=0,1,\ldots,d$.
\end{lem}
\begin{proof}
Let $\{v_i\}_{i=0}^d$ denote the basis for $R_d(a,b,c)$ from Proposition \ref{prop:UAWd}(i). Let
$$
w_i=(-1)^i\sum_{h=0}^{i}\binom{i}{h}(d-i+1)_{h}\textstyle(a+b+c+\frac{d}{2}-i+2)_{h}v_{i-h}
$$
for all $i=0,1, \ldots,d$.
Note that $w_i$ is a linear combination of $v_0,v_1, \ldots,v_i$ and the coefficient of $v_i$ in $w_i$ is nonzero for all $i=0,1,\ldots,d$. Hence $\{w_i\}_{i=0}^d$ is a basis for $R_d(a,b,c)$. A routine but tedious calculation yields that
\begin{align*}
(C-\theta^{\e}_i)w_i
&=
\left\{
\begin{array}{ll}
w_{i+1}
\qquad
&
\hbox{for $i=0,1,\ldots,d-1$},
\\
0
\qquad
&
\hbox{for $i=d$}.
\end{array}
\right.
\end{align*}
The lemma follows.
\end{proof}

The finite-dimensional irreducible $\Re$-modules are related to Leonard pairs and Leonard triples in the following ways:

\begin{prop}
\label{lem:lp}
Let $d\geq 0$ denote an integer and let $a,b,c\in \F$. Suppose that the $\Re$-module $R_d(a,b,c)$ is irreducible. Then the following are equivalent:
\begin{enumerate}
\item $A,B$ {\rm (}resp. $B,C${\rm )} {\rm (}resp. $C,A${\rm )} are diagonalizable on $R_d(a,b,c)$.

\item $A,B$ {\rm (}resp. $B,C${\rm )} {\rm (}resp. $C,A${\rm )} are multiplicity-free on $R_d(a,b,c)$.

\item $A,B$ {\rm (}resp. $B,C${\rm )} {\rm (}resp. $C,A${\rm )} act as a Leonard pair on $R_d(a,b,c)$.
\end{enumerate}
\end{prop}
\begin{proof}
(i) $\Leftrightarrow$ (ii): Immediate from Lemma \ref{lem:dia_UAW}.

(iii) $\Rightarrow$ (i): Immediate from Definition \ref{defn:lp}.

(ii) $\Rightarrow$ (iii): 
Without loss of generality we assume that $A,B$ are multiplicity-free on $R_d(a,b,c)$ and we show that $A,B$ act as a Leonard pair on $R_d(a,b,c)$. 

Let
$\theta_i
=\textstyle
\left(
a+\frac{d}{2}-i
\right)
\left(
a+\frac{d}{2}-i+1
\right)$ for every integer $i$. 
By Proposition \ref{prop:UAWd} the scalars $\{\theta_i\}_{i=0}^d$ are the eigenvalues of $A$ in $R_d(a,b,c)$. By Lemma \ref{lem:dia_UAW} the eigenvalues $\{\theta_i\}_{i=0}^d$ of $A$ are mutually distinct.
Let $v_i$ denote the $\theta_i$-eigenvector of $A$ in $R_d(a,b,c)$ for $i = 0,1,\ldots,d$. 
Note that $\{v_i\}_{i=0}^d$ form a basis for $R_d(a,b,c)$. Given any element $X \in \Re$ let $[X]$ denote the matrix representing $X$ with respect to $\{v_i\}_{i=0}^d$. By construction the matrix $[A]$ is
$$
\begin{pmatrix}
\theta_0 & & &  &{\bf 0}
\\ 
&\theta_1 
\\
& &\theta_2
 \\
& & &\ddots
 \\
{\bf 0} & & & &\theta_d
\end{pmatrix}.
$$
Let $i$ be an integer with $0\leq i \leq d$. Applying $v_i$ to either side of (\ref{AB}) yields $(A-\theta_{i-1})(A-\theta_{i+1})Bv_i$ is a scalar multiple of $v_i$. Hence
$$
(A-\theta_{i-1})(A-\theta_{i})(A-\theta_{i+1})Bv_i=0.
$$
It follows that $[B]$ is a tridiagonal matrix. Since the $\Re$-module $R_d(a,b,c)$ is irreducible, the tridiagonal matrix $[B]$ is irreducible.

Let 
$\theta^*_i 
= \textstyle
\left(
b+\frac{d}{2}-i
\right)
\left(
b+\frac{d}{2}-i+1
\right)$
for every integer $i$. 
By Proposition \ref{prop:UAWd} the scalars $\{\theta_i^*\}_{i=0}^d$ are the eigenvalues of $B$ in $R_d(a,b,c)$. By Lemma \ref{lem:dia_UAW} the eigenvalues $\{\theta_i^*\}_{i=0}^d$ of $B$ are mutually distinct.
Let $w_i$ denote the $\theta_i^*$-eigenvector of $B$ in $R_d(a,b,c)$ for $i = 0,1,\ldots,d$. Note that $\{w_i\}_{i=0}^d$ form a basis for $R_d(a,b,c)$. 
Given any element $X \in \Re$ let $\langle X\rangle$ denote the matrix representing $X$ with respect to $\{w_i\}_{i=0}^d$. By construction the matrix $\langle B\rangle$ is
$$
\begin{pmatrix}
\theta^*_0 & & &  &{\bf 0}
\\ 
&\theta^*_1 
\\
& &\theta^*_2
 \\
& & &\ddots
 \\
{\bf 0} & & & &\theta^*_d
\end{pmatrix}.
$$
Let $i$ be an integer with $0\leq i \leq d$. Applying $w_i$ to either side of (\ref{BA}) yields $(B-\theta_{i-1}^*)(B-\theta_{i+1}^*)Aw_i$ is a scalar multiple of $w_i$. Hence
$$
(B-\theta_{i-1}^*)(B-\theta_{i}^*)(B-\theta_{i+1}^*)Aw_i=0.
$$
It follows that $\langle A\rangle$ is a tridiagonal matrix. Since the $\Re$-module $R_d(a,b,c)$ is irreducible, the tridiagonal matrix $\langle A\rangle$ is irreducible. 
Therefore (iii) follows.
\end{proof}

\begin{prop}
\label{lem:lt}
Let $d\geq 0$ denote an integer and let $a,b,c$ denote scalars in $\F$. Suppose that the $\Re$-module $R_d(a,b,c)$ is irreducible. Then the following are equivalent:
\begin{enumerate}
\item $A,B,C$ are diagonalizable on $R_d(a,b,c)$.

\item $A,B,C$ are multiplicity-free on $R_d(a,b,c)$.

\item $A,B,C$ act as a Leonard triple on $R_d(a,b,c)$.
\end{enumerate}
\end{prop}
\begin{proof}
(i) $\Leftrightarrow$ (ii): Immediate from Lemma \ref{lem:dia_UAW}.

(iii) $\Rightarrow$ (i): Immediate from Definition \ref{defn:lt}.

(ii) $\Rightarrow$ (iii): In view of Lemma \ref{presentation2}, the part follows by an argument similar to the proof of Proposition \ref{lem:lp}.
\end{proof}

\section{Preliminaries on the even-dimensional irreducible $\H$-modules}\label{Pre:Hmodule}

\begin{prop}
[\!\!\cite{Huang:BImodule}]
\label{prop:E}
For any scalars $a,b,c\in \F$ and any odd integer $d\geq 1$, there exists a $(d+1)$-dimensional $\H$-module $E_d(a,b,c)$ satisfying the following conditions:
\begin{enumerate}
\item There exists a basis $\{v_i\}_{i=0}^d$ for $E_d(a,b,c)$ such that 
\begin{align*}
t_0 v_i
&=
\left\{
\begin{array}{ll}
\textstyle
\displaystyle i(d-i+1)v_{i-1}-\frac{d-2i+1}{2}v_i
\qquad
&
\hbox{for $i=2,4,\ldots,d-1$},
\\
\displaystyle \frac{d-2i-1}{2}v_i+v_{i+1}
\qquad
&
\hbox{for $i=1,3,\ldots,d-2$},
\end{array}
\right.
\\
t_0 v_0&=
-\frac{d+1}{2}v_0,
\qquad
t_0 v_d=
-\frac{d+1}{2}v_d,
\\
t_1 v_i
&=
\left\{
\begin{array}{ll}
i(i-d-1)v_{i-1}+av_i+v_{i+1}
\qquad
&
\hbox{for $i=2,4,\ldots,d-1$},
\\
-a v_i
\qquad
&
\hbox{for $i=1,3,\ldots,d$},
\end{array}
\right.
\\
t_1 v_0 &=
a v_0+ v_1,
\\
t_2 v_i
&=
\left\{
\begin{array}{ll}
b v_i
\qquad
&
\hbox{for $i=0,2,\ldots,d-1$},
\\
\textstyle
-(\sigma+i)(\tau+i)
v_{i-1}-b v_i-v_{i+1}
\qquad
&
\hbox{for $i=1,3,\ldots,d-2$},
\end{array}
\right.
\\
t_2 v_d &=
-(\sigma+d)(\tau+d) v_{d-1}-b v_d,
\\
t_3 v_i 
&=
\left\{
\begin{array}{ll}
\displaystyle -\frac{\sigma+\tau+2i+2}{2}v_i-v_{i+1}
\qquad
&
\hbox{for $i=0,2,\ldots,d-1$},
\\
\displaystyle (\sigma+i)(\tau+i)v_{i-1}+\frac{\sigma+\tau+2i}{2}v_i
\qquad
&
\hbox{for $i=1,3,\ldots,d$},
\end{array}
\right.
\end{align*}
where
\begin{gather*}
\sigma =
a+b+c-\frac{d+1}{2} ,
\qquad 
\tau =
a+b-c-\frac{d+1}{2}.
\end{gather*}

\item The elements $t_0^{2}, t_1^{2},t_2^{2},t_3^{2}$ act on $E_d(a,b,c)$ as scalar multiplication by
$
\frac{(d+1)^{2}}{4},
a^{2},
b^{2},
c^{2}
$
respectively.
\end{enumerate}
\end{prop}

\begin{thm}
[Theorem 2.5, \cite{Huang:BImodule}]
\label{thm:irr_E}
Let $a,b,c\in \F$ and let $d\geq 1$ denote an odd integer. Then the $\H$-module $E_d(a,b,c)$ is irreducible if and only if
$$
a+b+c,-a+b+c,a-b+c,a+b-c
\not\in 
\left\{
\frac{d-1}{2}-i\,\bigg |\, i=0,2,\ldots,d-1
\right\}.
$$
\end{thm}

Recall that $\{\pm 1\}$ is a group under multiplication and the group $\{\pm 1\}^2$ is isomorphic to the Klein $4$-group.
By Definition \ref{defn:H}  there exists a unique $\{\pm 1\}^2$-action on $\H$ such that each $\e \in \{\pm 1\}^2$ acts on $\H$ as an algebra automorphism in the following way:

\begin{table}[H]
\centering
\extrarowheight=3pt
\begin{tabular}{c|rrrr}
$u$  &$t_0$ &$t_1$ &$t_2$ &$t_3$ 
\\

\midrule[1pt]

${u^{(1,1)}}$ &$t_0$  &$t_1$ &$t_2$ &$t_3$ 
\\
${u^{(1,-1)}}$ &$t_1$ &$t_0$ &$t_3$ &$t_2$ 
\\
${u^{(-1,1)}}$ &$t_2$ &$t_3$ &$t_0$ &$t_1$
\\
${u^{(-1,-1)}}$ &$t_3$ &$t_2$ &$t_1$ &$t_0$
\end{tabular}
\caption{The $\{\pm 1\}^{2}$-action on $\H$}\label{Z/4Z-action}
\end{table}

Let $V$ denote an $\H$-module. For any algebra automorphism $\e$ of $\H$ the notation 
$
V^\e
$ 
stands for an alternate $\H$-module structure on $V$ given by 
$$
xv:=x^\e v
\qquad \hbox{for all $x\in \H$ and all $v\in V$}.
$$

\begin{thm}
[Theorem 6.3, \cite{Huang:BImodule}]
\label{thm:onto_E}
Let $d\geq 1$ denote an odd integer. Suppose that $V$ is a $(d+1)$-dimensional irreducible $\H$-module. Then there exist $a,b,c\in \F$ and $\e\in \{\pm 1\}^2$ such that the $\H$-module $E_d(a,b,c)^\e$ is isomorphic to $V$.
\end{thm}

\begin{thm}
[Theorem 5.3, \cite{Huang:BImodule}]
\label{thm:iso_E}
Let $a,b,c\in \F$ and let $d\geq 1$ denote an odd integer. Suppose that the $\H$-module $E_d(a,b,c)$ is irreducible. Then the following hold:
\begin{enumerate}
\item The $\H$-module $E_d(a,b,c)$ is isomorphic to $E_d(-a,b,c)$.

\item The $\H$-module $E_d(a,b,c)$ is isomorphic to $E_d(a,-b,c)$.

\item The $\H$-module $E_d(a,b,c)$ is isomorphic to $E_d(a,b,-c)$.
\end{enumerate}
\end{thm}

\begin{lem}
\label{lem:det_E}
Let $a,b,c\in \F$ and let $d\geq 1$ denote an odd integer. Then the traces of $t_0,t_1,t_2,t_3$ on $E_d(a,b,c)$ are $-(d+1),0,0,0$ respectively.
\end{lem}
\begin{proof}
It is routine to verify the lemma by using Proposition \ref{prop:E}(i).\end{proof}

By means of Theorem \ref{thm:iso_E} and Lemma \ref{lem:det_E} we develop the following discriminant to determine the scalars $a,b,c \in \F$ and $\e \in \{\pm 1\}^2$ in Theorem \ref{thm:onto_E}:

\begin{thm}
\label{thm:onto2_E}
Let $d\geq 1$ denote an odd integer. Suppose that $V$ is a $(d+1)$-dimensional irreducible $\H$-module. For any scalars $a,b,c \in \F$ and any $\e\in \{\pm 1\}^{2}$ the following are equivalent:
\begin{enumerate}
\item The $\H$-module $E_d(a,b,c)^\e$ is isomorphic to $V$.
\item The trace of $t_0$  on $V^{\e}$ is $-(d+1)$ and $t_1^{2},t_2^{2},t_3^{2}$ act on $V^{\e}$ as scalar multiplication by $a^{2},b^{2},c^{2}$ respectively.
\end{enumerate}
\end{thm}
\begin{proof}
(i) $\Rightarrow$ (ii): By (i) the $\H$-module $E_d(a,b,c)$ is isomorphic to $V^{\e}$. Hence (ii) follows by Proposition \ref{prop:E}(ii) and Lemma \ref{lem:det_E}.

(ii) $\Rightarrow$ (i): By Theorem \ref{thm:onto_E} there are an $\e'\in \{\pm 1\}^{2}$ and $a',b',c'\in \F$ such that the $\H$-module $E_d(a',b',c')^{\e'}$ is isomorphic to $V$. Hence the $\H$-module $E_d(a',b',c')$ is isomorphic to $V^{\e'}$. By Lemma \ref{lem:det_E} the traces of $t_0,t_1,t_2,t_3$ on $V^{\e'}$ are $-(d+1),0,0,0$ respectively. Since the trace of $t_0$ on $V^{\e}$ is $-(d+1)$, it follows from Table \ref{Z/4Z-action} that $\e=\e'$.
Combined with Proposition \ref{prop:E}(ii) this yields that $a',b',c'$ are $\pm a,\pm b,\pm c$ respectively. 
Now (i) follows from Theorem \ref{thm:iso_E}.
\end{proof}

\section{The conditions for $A,B,C$ as diagonalizable on even-dimensional irreducible $\H$-modules}\label{s:even}

For convenience we adopt the following conventions in this section: Let $d\geq 1$ denote an odd integer. Let $a,b,c$ denote any scalars in $\F$. Let $\{v_i\}_{i=0}^d$ denote the basis for $E_d(a,b,c)$ from Proposition \ref{prop:E}(i).

\begin{lem}
\label{lem:t0t1_E}
The action of $t_0+t_1$ on $E_d(a,b,c)$ is as follows:
\begin{align*}
\left(
t_0+t_1+(-1)^i\left(
\frac{d}{2}-a-i
\right)
+\frac{1}{2}
\right)
v_i
=\left\{
\begin{array}{ll}
v_{i+1} 
\qquad 
&\hbox{for $i=0,1,\ldots,d-1$},
\\
0
\qquad 
&\hbox{for $i=d$}.
\end{array}
\right.
\end{align*}
\end{lem}
\begin{proof}
It is routine to verify the lemma by using Proposition \ref{prop:E}(i).
\end{proof}

\begin{lem}\label{lem:t0t2_E}
Suppose that the $\H$-module $E_d(a,b,c)$ is irreducible. Then the following hold:
\begin{enumerate}
\item There exists a basis $\{w_i\}_{i=0}^d$ for $E_d(a,b,c)$ such that 
\begin{align*}
\left(
t_0+t_2+(-1)^i\left(\frac{d}{2}-b-i\right)+\frac{1}{2}
\right)
w_i
=\left\{
\begin{array}{ll}
w_{i+1} 
\qquad 
&\hbox{for $i=0,1,\ldots,d-1$},
\\
0
\qquad 
&\hbox{for $i=d$}.
\end{array}
\right.
\end{align*}

\item There exists a basis $\{w_i\}_{i=0}^d$ for  $E_d(a,b,c)$ such that 
\begin{align*}
\left(
t_0+t_3+(-1)^i\left(\frac{d}{2}-c-i\right)+\frac{1}{2}
\right)
w_i
=\left\{
\begin{array}{ll}
w_{i+1} 
\qquad 
&\hbox{for $i=0,1,\ldots,d-1$},
\\
0
\qquad 
&\hbox{for $i=d$}.
\end{array}
\right.
\end{align*}
\end{enumerate}
\end{lem}
\begin{proof}
Let $V$ denote the $\H$-module $E_d(a,b,c)$.

(i): By Definition \ref{defn:H} there exists a unique algebra automorphism $\rho:\H\to \H$ given by 
\begin{eqnarray}\label{auto_t0t2}
(t_0,t_1,t_2,t_3) 
&\mapsto &
(t_0,t_2,t_3,t_1)
\end{eqnarray}
whose inverse sends 
$(t_0,t_1,t_2,t_3)$ to
$(t_0,t_3 ,t_1,t_2)$. By Proposition \ref{prop:E}(ii) the elements $t_0^{2}$, $t_1^{2}$, $t_2^{2}$, $t_3^{2}$ act on $V^\rho$ as scalar multiplication by 
$
\frac{(d+1)^2}{4},b^2,c^2,a^2
$ 
respectively.
By Lemma \ref{lem:det_E} the trace of $t_0$ on $V^\rho$ is $-(d+1)$. Therefore the $\H$-module $V^\rho$ is isomorphic to 
$
E_d(b,c,a)
$
by Theorem \ref{thm:onto2_E}. 
It follows from Lemma \ref{lem:t0t1_E} that there exists a basis $\{w_i\}_{i=0}^d$ for $V^\rho$ such that 
\begin{align*}
\left(
t_0+t_1+(-1)^i\left(\frac{d}{2}-b-i\right)+\frac{1}{2}
\right)
w_i
=\left\{
\begin{array}{ll}
w_{i+1} 
\qquad 
&\hbox{for $i=0,1,\ldots,d-1$},
\\
0
\qquad 
&\hbox{for $i=d$}.
\end{array}
\right.
\end{align*}
It follows from (\ref{auto_t0t2}) that the action of $t_0+t_1$ on $V^\rho$ is identical to the action of $t_0+t_2$ on $V$. Hence (i) follows.

(ii): By Definition \ref{defn:H} there exists a unique algebra automorphism $\rho:\H\to \H$ given by 
\begin{eqnarray}\label{auto_t3t0}
(t_0,t_1,t_2,t_3) 
&\mapsto &
(t_0,t_3,t_2,t_1)
\end{eqnarray}
whose inverse sends 
$(t_0,t_1,t_2,t_3)$ to
$(t_0,t_3,t_2,t_1)$. By Proposition \ref{prop:E}(ii) the elements $t_0^{2}$, $t_1^{2}$, $t_2^{2}$, $t_3^{2}$ act on $V^\rho$ as scalar multiplication by 
$
\frac{(d+1)^2}{4},c^2,b^2,a^2
$
respectively.
By Lemma \ref{lem:det_E} the trace of $t_0$ on $V^\rho$ is $-(d+1)$. Therefore the $\H$-module $V^\rho$ is isomorphic to 
$
E_d(c,b,a)
$
by Theorem \ref{thm:onto2_E}.
It follows from Lemma \ref{lem:t0t1_E} that there exists a basis $\{w_i\}_{i=0}^d$ for $V^\rho$ such that 
\begin{align*}
\left(
t_0+t_1+(-1)^i\left(\frac{d}{2}-c-i\right)+\frac{1}{2}
\right)
w_i
=\left\{
\begin{array}{ll}
w_{i+1} 
\qquad 
&\hbox{for $i=0,1,\ldots,d-1$},
\\
0
\qquad 
&\hbox{for $i=d$}.
\end{array}
\right.
\end{align*}
It follows from (\ref{auto_t3t0}) that the action of $t_0+t_1$ on $V^\rho$ is identical to the action of $t_0+t_3$ on $V$. Hence (ii) follows.
\end{proof}


\begin{lem}\label{lem:t0+t1_vanish}
Let $j$ denote an integer with $0 \le j \le d$. Then
\begin{eqnarray} \label{v0vanish}
\prod_{\substack{i=0 \\ i\not=j}}^d \left(t_0+t_1+(-1)^i\left(\frac{d}{2}-a-i\right)+\frac{1}{2}\right) v_0 \not=0.
\end{eqnarray}
\end{lem}
\begin{proof}
It is immediate from Lemma \ref{lem:t0t1_E} that the left--hand side of (\ref{v0vanish}) is equal to 
\begin{eqnarray} \label{vjrepeat}
\prod_{i=j+1}^d \left(t_0+t_1+(-1)^i\left(\frac{d}{2}-a-i\right)+\frac{1}{2}\right) v_j.
\end{eqnarray}
If $j=d$ then (\ref{vjrepeat}) is equal to $v_d$. Suppose that $j \not= d$. Applying Lemma \ref{lem:t0t1_E} yields that (\ref{vjrepeat}) is equal to $v_d$ plus a linear combination of $v_0,v_1,\ldots,v_{d-1}$.
Hence (\ref{vjrepeat}) is nonzero.
The lemma follows.
\end{proof}

\begin{lem}\label{lem:t1t0_diag}
Suppose that the $\H$-module $E_d(a,b,c)$ is irreducible. Then the following are equivalent:
\begin{enumerate}
\item $t_0+t_1$ is diagonalizable on $E_d(a,b,c)$.
\item $t_0+t_1$ is multiplicity-free on $E_d(a,b,c)$.
\item $2a$ is not among $d-1,d-3,\ldots,1-d$.
\end{enumerate}
\end{lem}
\begin{proof}
By Lemma \ref{lem:t0t1_E} this characteristic polynomial of $t_0+t_1$ in $E_d(a,b,c)$ has the roots
\begin{align}
&(-1)^{i-1}\left(\frac{d}{2}-a-i\right)-\frac{1}{2} 
\label{eigen_t0t1_E}
\end{align}
for $i=0,1,\ldots,d$.

(ii) $\Leftrightarrow$ (iii): Since $\F$ is of characteristic zero, the scalars (\ref{eigen_t0t1_E}) for $i=0,2,\ldots,d-1$ are mutually distinct and the scalars (\ref{eigen_t0t1_E}) for $i=1,3,\ldots,d$ are mutually distinct. 
Hence the scalars (\ref{eigen_t0t1_E}) for all $i=0,1,\ldots,d$ are mutually distinct if and only if (iii) holds. Therefore (ii) and (iii) are equivalent.

(ii), (iii) $\Rightarrow$ (i): Trivial.

(i) $\Rightarrow$ (ii), (iii): By Lemma \ref{lem:t0t1_E} the product 
\begin{gather*}
\prod_{i=0}^d 
\left(
t_0+t_1+(-1)^i\left(\frac{d}{2}-a-i\right)+\frac{1}{2}
\right)
\end{gather*}
vanishes at $v_0$. Combined with Lemma \ref{lem:t0+t1_vanish} the $(t_0+t_1)$-annihilator of $v_0$ is equal to the characteristic polynomial of $(t_0+t_1)$ in $E_d(a,b,c)$. Therefore (i) implies (ii) and (iii).
\end{proof}

\begin{lem}\label{lem:t2t0_diag}
Suppose that the $\H$-module $E_d(a,b,c)$ is irreducible. Then the following are equivalent:
\begin{enumerate}
\item $t_0+t_2$ is diagonalizable on $E_d(a,b,c)$.

\item $t_0+t_2$ is multiplicity-free on $E_d(a,b,c)$.

\item $2b$ is not among $d-1,d-3,\ldots,1-d$.
\end{enumerate}
\end{lem}
\begin{proof}
In view of Lemma \ref{lem:t0t2_E}(i), the lemma follows by an argument similar to the proof of Lemma \ref{lem:t1t0_diag}.
\end{proof}

\begin{lem}\label{lem:t3t0_diag}
Suppose that the $\H$-module $E_d(a,b,c)$ is irreducible. Then the following are equivalent:
\begin{enumerate}
\item $t_0+t_3$ is diagonalizable on $E_d(a,b,c)$.

\item $t_0+t_3$ is multiplicity-free on $E_d(a,b,c)$.

\item $2c$ is not among $d-1,d-3,\ldots,1-d$.
\end{enumerate}
\end{lem}
\begin{proof}
In view of Lemma \ref{lem:t0t2_E}(ii), the lemma follows by an argument similar to the proof of Lemma \ref{lem:t1t0_diag}.
\end{proof}

\begin{lem}\label{lem:t1t3+t1t3inv_diag}
Suppose that the $\H$-module $E_d(a,b,c)$ is irreducible. Then the following hold:
\begin{enumerate}

\item If $2a$ is not among $d-3,d-5,\ldots,3-d$, then 
$
(t_2+t_3-1)(t_2+t_3+1)
$
is diagonalizable on $E_d(a,b,c)$.

\item If $2b$ is not among $d-3,d-5,\ldots,3-d$, then 
$
(t_1+t_3-1)(t_1+t_3+1)
$
is diagonalizable on $E_d(a,b,c)$.

\item If $2c$ is not among $d-3,d-5,\ldots,3-d$, then 
$
(t_1+t_2-1)(t_1+t_2+1)
$
is diagonalizable on $E_d(a,b,c)$.

\end{enumerate}
\end{lem}
\begin{proof}
(i): For convenience we write $L=(t_2+t_3-1)(t_2+t_3+1)$. Given any element $X$ of $\H$, let $[X]$ denote the matrix representing $X$ with respect to $\{v_i\}_{i=0}^d$. 
Using (\ref{t0123}) yields that
\begin{gather*}
L=(t_0+t_1)(t_0+t_1+2).
\end{gather*}
Combined with Lemma \ref{lem:t0t1_E} yields that the matrix $[L]$ is a lower triangular matrix of the form
\begin{gather}\label{[D]}
\begin{pmatrix}
\theta_0 & & & & & &{\bf 0}
\\
0 &\theta_0
\\
  & &\theta_1
\\
  &  & &\theta_1
\\
  & & & &\ddots
\\
  & & & & &\theta_{\frac{d-1}{2}}
\\  
*  & & & & &0 &\theta_{\frac{d-1}{2}}
\end{pmatrix}
\end{gather}
where 
$$
\theta_i=
\left(a-\frac{d+1}{2}+2i\right)\left(a-\frac{d+1}{2}+2i+2\right)
\qquad 
\hbox{for $i=0,1,\ldots,\displaystyle \frac{d-1}{2}$}.
$$

Since $2a$ is not among $d-3,d-5,\ldots,3-d$, the scalars $\{\theta_i\}_{i=0}^{\frac{d-1}{2}}$ are mutually distinct. Hence the $\theta_i$-eigenspace of $L$ in $E_d(a,b,c)$ has dimension less than or equal to two for all $i=0,1,\ldots,\frac{d-1}{2}$. By (\ref{[D]}) the first two rows of $L-\theta_0$ are zero and the last two columns of $L-\theta_{\frac{d-1}{2}}$ are zero. Hence the $\theta_i$-eigenspace of $L$ in $E_d(a,b,c)$ has dimension two for $i=0,\frac{d-1}{2}$.
To see the diagonalizability of $L$ it remains to show that the $\theta_i$-eigenspace of $L$ in $E_d(a,b,c)$ has dimension two for all $i=1,2,\ldots,\frac{d-3}{2}$.

By Lemma \ref{lem:t0t1_E} the matrix $ [t_0+t_1]$ is a lower triangular matrix of the form 
\begin{gather}\label{[t1t3]}
\begin{pmatrix}
\vartheta_0 & & & & & &{\bf 0}
\\
 &-\vartheta_1
\\
  & &\vartheta_1
\\
  &  & &-\vartheta_2
\\
  & & & &\ddots
\\
  & & & & &\vartheta_{\frac{d-1}{2}}
\\  
*  & & & & & &-\vartheta_{\frac{d+1}{2}}
\end{pmatrix}
\end{gather}
where 
$$
\vartheta_i=a-\frac{d+1}{2}+2i
\qquad 
\hbox{for $i=0,1,\ldots,\displaystyle \frac{d+1}{2}$}.
$$

Let $i \in\{1,2,..,\frac{d-3}{2}\}$ be given. Let $u$ and $w$ denote the eigenvectors of $t_0+t_1$ in $E_d(a,b,c)$ corresponding to the eigenvalues $\vartheta_i$ and $-\vartheta_{i+1}$, respectively. Since $2a$ is not among $d-5,d-9,\ldots,5-d$, the scalars $\vartheta_i$ and $-\vartheta_{i+1}$ are distinct. It follows that $u$ and $w$ are linearly independent. Observe that $u$ and $w$ are the $\theta_i$-eigenvectors of $L$. Hence the $\theta_i$-eigenspace of $L$ in $E_d(a,b,c)$ has dimension two. The statement (i) follows.

(ii), (iii): Using Lemma \ref{lem:t0t2_E}(i), (ii) the statements (ii), (iii) follow by the arguments similar to the proof for (i).
\end{proof}

Recall the finite-dimensional irreducible $\Re$-modules from \S\ref{s:AWmodule} and the even-dimensional irreducible $\H$-modules from \S\ref{Pre:Hmodule}.

\begin{thm}
[\S 4.2--\S 4.5, \cite{Huang:R<BImodules}]
\label{thm:CF_even}
Suppose that the $\H$-module $E_d(a,b,c)$ is irreducible. Then the following hold:
\begin{enumerate}
\item If $d=1$ then the $\Re$-module $E_d(a,b,c)$ is irreducible and it is isomorphic to 
\begin{align*}
R_1\left(-\frac{a+1}{2},-\frac{b+1}{2},-\frac{c+1}{2}\right).
\end{align*}

\item If $d\geq 3$ then the factors of any composition series for the $\Re$-module $E_d(a,b,c)$ are isomorphic to 
\begin{align*}
&R_\frac{d+1}{2}\left(-\frac{a+1}{2},-\frac{b+1}{2},-\frac{c+1}{2}\right),
\\
&R_\frac{d-3}{2}\left(-\frac{a+1}{2},-\frac{b+1}{2},-\frac{c+1}{2}\right).
\end{align*}

\item The factors of any composition series for the $\Re$-module $E_d(a,b,c)^{(1,-1)}$ are isomorphic to
\begin{align*}
&R_\frac{d-1}{2}\left(-\frac{a}{2},-\frac{b+1}{2},-\frac{c+1}{2}\right),
\\
&R_\frac{d-1}{2}\left(-\frac{a}{2}-1,-\frac{b+1}{2},-\frac{c+1}{2}\right).
\end{align*}

\item The factors of any composition series for the $\Re$-module $E_d(a,b,c)^{(-1,1)}$ are isomorphic to
\begin{align*}
&R_\frac{d-1}{2}\left(-\frac{a+1}{2},-\frac{b}{2},-\frac{c+1}{2}\right),
\\
&R_\frac{d-1}{2}\left(-\frac{a+1}{2},-\frac{b}{2}-1,-\frac{c+1}{2}\right).
\end{align*}

\item The factors of any composition series for the $\Re$-module $E_d(a,b,c)^{(-1,-1)}$ are isomorphic to
\begin{align*}
&R_\frac{d-1}{2}\left(-\frac{a+1}{2},-\frac{b+1}{2},-\frac{c}{2}\right),
\\
&R_\frac{d-1}{2}\left(-\frac{a+1}{2},-\frac{b+1}{2},-\frac{c}{2}-1\right).
\end{align*}
\end{enumerate}
\end{thm}

We now give the necessary and sufficient conditions for $A,B,C$ to be multiplicity-free on the composition factors of even-dimensional irreducible $\H$-modules.

\begin{lem}\label{lem:diag_ABC_E}
Suppose that the $\H$-module $E_d(a,b,c)$ is irreducible. Then the following hold:
\begin{enumerate}
\item $A$ 
is multiplicity-free on all composition factors of the $\Re$-module $E_d(a,b,c)$ if and only if $2a$ is not among $d-1,d-3,\ldots,1-d$.

\item $B$
is multiplicity-free on all composition factors of the $\Re$-module $E_d(a,b,c)$ if and only if $2b$ is not among $d-1,d-3,\ldots,1-d$.

\item $C$
is multiplicity-free on all composition factors of the $\Re$-module $E_d(a,b,c)$ if and only if $2c$ is not among $d-1,d-3,\ldots,1-d$.
\end{enumerate}
\end{lem}
\begin{proof}
Immediate from Lemma \ref{lem:dia_UAW} and Theorem \ref{thm:CF_even}(i), (ii).
\end{proof}

\begin{lem}\label{lem:diag_ABC_E1}
Suppose that the $\H$-module $E_d(a,b,c)$ is irreducible. Then the following hold:
\begin{enumerate}
\item $A$ 
is multiplicity-free on all composition factors of the $\Re$-module $E_d(a,b,c)^{(1,-1)}$ if and only if 
$2a$ is not among $d-1,d-3,\ldots,1-d$.

\item $B$
is multiplicity-free on all composition factors of the $\Re$-module $E_d(a,b,c)^{(1,-1)}$ if and only if 
$2b$ is not among $d-3,d-5,\ldots,3-d$.

\item $C$ 
is multiplicity-free on all composition factors of the $\Re$-module $E_d(a,b,c)^{(1,-1)}$ if and only if 
$2c$ is not among $d-3,d-5,\ldots,3-d$.
\end{enumerate}
\end{lem}
\begin{proof}
Immediate from Lemma \ref{lem:dia_UAW} and Theorem \ref{thm:CF_even}(iii).
\end{proof}

\begin{lem}\label{lem:diag_ABC_E2}
Suppose that the $\H$-module $E_d(a,b,c)$ is irreducible. Then the following hold:
\begin{enumerate}
\item $A$ 
is multiplicity-free on all composition factors of the $\Re$-module $E_d(a,b,c)^{(-1,1)}$ if and only if 
$2a$ is not among $d-3,d-5,\ldots,3-d$.

\item $B$
is multiplicity-free on all composition factors of the $\Re$-module $E_d(a,b,c)^{(-1,1)}$ if and only if 
$2b$ is not among $d-1,d-3,\ldots,1-d$.

\item $C$ 
is multiplicity-free on all composition factors of the $\Re$-module $E_d(a,b,c)^{(-1,1)}$ if and only if 
$2c$ is not among $d-3,d-5,\ldots,3-d$.
\end{enumerate}
\end{lem}
\begin{proof}
Immediate from Lemma \ref{lem:dia_UAW} and Theorem \ref{thm:CF_even}(iv).
\end{proof}

\begin{lem}\label{lem:diag_ABC_E3}
Suppose that the $\H$-module $E_d(a,b,c)$ is irreducible. Then the following hold:
\begin{enumerate}
\item $A$ 
is multiplicity-free on all composition factors of the $\Re$-module $E_d(a,b,c)^{(-1,-1)}$ if and only if 
$2a$ is not among $d-3,d-5,\ldots,3-d$.

\item $B$
is multiplicity-free on all composition factors of the $\Re$-module $E_d(a,b,c)^{(-1,-1)}$ if and only if 
$2b$ is not among $d-3,d-5,\ldots,3-d$.

\item $C$ 
is multiplicity-free on all composition factors of the $\Re$-module $E_d(a,b,c)^{(-1,-1)}$ if and only if 
$2c$ is not among $d-1,d-3,\ldots,1-d$.
\end{enumerate}
\end{lem}
\begin{proof}
Immediate from Lemma \ref{lem:dia_UAW} and Theorem \ref{thm:CF_even}(v).
\end{proof}

We are in the position to prove Theorem \ref{thm:diagonal} in the even-dimensional case.

\begin{thm}\label{thm:diagonal_even}
Suppose that $V$ is an even-dimensional irreducible $\H$-module. Then the following are equivalent:
\begin{enumerate}
\item $A$ {\rm (}resp. $B${\rm )} {\rm (}resp. $C${\rm )}
is diagonalizable on $V$.

\item $A$ {\rm (}resp. $B${\rm )} {\rm (}resp. $C${\rm )} is diagonalizable on all composition factors of the $\Re$-module $V$.

\item $A$ {\rm (}resp. $B${\rm )} {\rm (}resp. $C${\rm )} is multiplicity-free on all composition factors of the $\Re$-module $V$.
\end{enumerate}
\end{thm}
\begin{proof}
(i) $\Rightarrow$ (ii): Trivial.

(ii) $\Rightarrow$ (iii): Immediate from Lemma \ref{lem:dia_UAW}.

(iii) $\Rightarrow$ (i): 
Suppose that (iii) holds. 
Let $d=\dim V-1$. 
By Theorem \ref{thm:onto_E} there exist $a,b,c\in \F$ and $\e\in \{\pm 1\}^2$ such that the $\H$-module $E_d(a,b,c)^\e$ is isomorphic to $V$.

Using Lemmas \ref{lem:t1t0_diag}--\ref{lem:t3t0_diag} and \ref{lem:diag_ABC_E} yields that (i) holds for the case $\e=(1,1)$. Note that the actions of $A,B,C$ on $E_d(a,b,c)^{(1,-1)}$ are identical to the actions of 
$$
\frac{(t_0+t_1-1)(t_0+t_1+1)}{4},
\qquad 
\frac{(t_1+t_3-1)(t_1+t_3+1)}{4},
\qquad 
\frac{(t_1+t_2-1)(t_1+t_2+1)}{4}
$$ 
on $E_d(a,b,c)$ respectively. Hence (i) holds for the case $\e=(1,-1)$ by Lemmas  \ref{lem:t1t0_diag}, \ref{lem:t1t3+t1t3inv_diag} and \ref{lem:diag_ABC_E1}. 
The actions of $A,B,C$ on $E_d(a,b,c)^{(-1,1)}$ are identical to the actions of 
$$
\frac{(t_2+t_3-1)(t_2+t_3+1)}{4},
\qquad 
\frac{(t_0+t_2-1)(t_0+t_2+1)}{4},
\qquad 
\frac{(t_1+t_2-1)(t_1+t_2+1)}{4}
$$ 
on $E_d(a,b,c)$ respectively. Hence (i) holds for the case $\e=(-1,1)$ by Lemmas  \ref{lem:t2t0_diag}, \ref{lem:t1t3+t1t3inv_diag} and \ref{lem:diag_ABC_E2}. The actions of $A,B,C$ on $E_d(a,b,c)^{(-1,-1)}$ are identical to the actions of 
$$
\frac{(t_2+t_3-1)(t_2+t_3+1)}{4},
\qquad 
\frac{(t_1+t_3-1)(t_1+t_3+1)}{4},
\qquad 
\frac{(t_0+t_3-1)(t_0+t_3+1)}{4}
$$ 
on $E_d(a,b,c)$ respectively. 
Hence (i) holds for the case $\e=(-1,-1)$ by Lemmas  \ref{lem:t3t0_diag}, \ref{lem:t1t3+t1t3inv_diag} and \ref{lem:diag_ABC_E3}. We have shown that (i) holds for all elements $\e\in \{\pm 1\}^2$. Therefore (i) follows.
\end{proof}

\section{Preliminaries on the odd-dimensional irreducible $\H$-modules}\label{Pre:Hmodule2}

\begin{prop}
[\!\! \cite{Huang:BImodule}]
\label{prop:O}
For any scalars $a,b,c\in \F$ and any even integer $d\geq 0$, there exists a $(d+1)$-dimensional $\H$-module $O_d(a,b,c)$ satisfying the following conditions:
\begin{enumerate}
\item There exists a basis $\{v_i\}_{i=0}^d$ for $O_d(a,b,c)$ such that 
\begin{align*}
t_0 v_i
&=
\left\{
\begin{array}{ll} 
\displaystyle -i(\sigma+i)v_{i-1}+\frac{\sigma+2i}{2}v_i
\qquad
&
\hbox{for $i=2,4,\ldots,d$},
\\
\displaystyle -\frac{\sigma+2i+2}{2}v_i+v_{i+1}
\qquad
&
\hbox{for $i=1,3,\ldots,d-1$},
\end{array}
\right.
\\
t_0 v_0&=
\frac{\sigma}{2}v_0,
\\
t_1 v_i
&=
\left\{
\begin{array}{ll}
\displaystyle i(\sigma+i)v_{i-1}+\frac{\lambda}{2}v_i+v_{i+1}
\qquad
&
\hbox{for $i=2,4,\ldots,d-2$},
\\
\displaystyle -\frac{\lambda}{2} v_i
\qquad
&
\hbox{for $i=1,3,\ldots,d-1$},
\end{array}
\right.
\\
t_1 v_0 &=
\frac{\lambda}{2} v_0+ v_1,
\qquad
t_1 v_d =
d(\sigma+d)v_{d-1}+\frac{\lambda}{2} v_d,
\\
t_2 v_i
&=
\left\{
\begin{array}{ll}
\displaystyle \frac{\nu}{2} v_i
\qquad
&
\hbox{for $i=0,2,\ldots,d$},
\\
\displaystyle (d-i-1)(\tau+i)
v_{i-1}-\frac{\nu}{2} v_i-v_{i+1}
\qquad
&
\hbox{for $i=1,3,\ldots,d-1$},
\end{array}
\right.
\\
t_3 v_i 
&=
\left\{
\begin{array}{ll}
\displaystyle \frac{2d+\mu-2i}{2}v_i-v_{i+1}
\qquad
&
\hbox{for $i=0,2,\ldots,d-2$},
\\
\displaystyle (i-d-1)(\tau+i)v_{i-1}-\frac{2d+\mu-2i+2}{2}v_i
\qquad
&
\hbox{for $i=1,3,\ldots,d-1$},
\end{array}
\right.
\\
t_3 v_d
&=
\frac{\mu}{2}v_d,
\end{align*}
where
\begin{align*}
\sigma &=
a+b+c-\frac{d+1}{2} ,
\qquad 
\tau =
a+b-c-\frac{d+1}{2} ,
\\
\lambda &=
a-b-c-\frac{d+1}{2} ,
\qquad
\mu =
c-a-b-\frac{d+1}{2} ,
\\
\nu &=
b-a-c-\frac{d+1}{2}.
\end{align*}

\item The elements $t_0^{2}, t_1^{2},t_2^{2},t_3^{2}$ act on $O_d(a,b,c)$ as scalar multiplication by
$$
\left(\frac{a+b+c}{2}-\frac{d+1}{4}\right)^{2},
\qquad
\left(\frac{a-b-c}{2}-\frac{d+1}{4}\right)^{2},
$$
$$
\left(\frac{c-a-b}{2}-\frac{d+1}{4}\right)^{2},
\qquad
\left(\frac{b-a-c}{2}-\frac{d+1}{4}\right)^{2},
$$
respectively.
\end{enumerate}
\end{prop}

\begin{thm}
[Theorem 2.8, \cite{Huang:BImodule}]
\label{thm:irr_O}
Let $d\geq 0$ denote an even integer.  Then the following hold:
\begin{enumerate}
\item For any $a, b, c\in \F$ the $\H$-module $O_d(a,b,c)$ is irreducible if and only if 
$$
a+b+c, a-b-c, -a+b-c, -a-b+c\not\in \left\{\frac{d+1}{2}-i \,\bigg|\, \hbox{$i =2,4,\ldots,d$}\right\}.
$$

\item Suppose that $V$ is a $(d+1)$-dimensional irreducible $\H$-module. Then there exist $a,b,c\in \F$ such that the $\H$-module $O_d(a,b,c)$ is isomorphic to $V$.
\end{enumerate}
\end{thm}

\begin{lem}
\label{lem:det_O}
Let $a,b,c \in \F$ and let $d\geq 0$ denote an even integer. Then the traces of $t_0+t_1+\frac{1}{2},t_0+t_2+\frac{1}{2},t_0+t_3+\frac{1}{2}$ on $O_d(a,b,c)$ are $a,b,c$ respectively.
\end{lem}
\begin{proof}
It is routine to verify the lemma by using Proposition \ref{prop:O}(i).
\end{proof}

By means of Lemma \ref{lem:det_O} we develop the following discriminant to determine the scalars $a,b,c\in \F$ in Theorem \ref{thm:irr_O}(ii):

\begin{thm}
\label{thm:onto2_O}
Let $d\geq 0$ denote an even integer. Suppose that $V$ is a $(d+1)$-dimensional irreducible $\H$-module. For any scalars $a,b,c\in \F$ the following are equivalent:
\begin{enumerate}
\item The $\H$-module $O_d(a,b,c)$ is isomorphic to $V$. 

\item The traces of $t_0+t_1+\frac{1}{2},t_0+t_2+\frac{1}{2},t_0+t_3+\frac{1}{2}$ on $V$ are $a,b,c$ respectively.
\end{enumerate}
\end{thm}
\begin{proof}
Immediate from Theorem \ref{thm:irr_O}(ii) and Lemma \ref{lem:det_O}.
\end{proof}

\section{The conditions for $A,B,C$ as diagonalizable on odd-dimensional irreducible $\H$-modules}\label{s:odd}

For convenience we adopt the following notations in this section: Let $d\geq 0$ denote an even integer. Let $a,b,c$ denote any scalars in $\F$. Let $\{v_i\}_{i=0}^d$ denote the basis for $O_d(a,b,c)$ from Proposition \ref{prop:O}(i).

\begin{lem}
\label{lem:t0t1_O}
The action of $t_0+t_1$ on $O_d(a,b,c)$ is as follows:
\begin{align*}
\left(t_0+t_1+(-1)^i\left(\frac{d}{2}-a-i\right)+\frac{1}{2}\right)
v_i
=\left\{
\begin{array}{ll}
v_{i+1} 
\qquad 
&\hbox{for $i=0,1,\ldots,d-1$},
\\
0
\qquad 
&\hbox{for $i=d$}.
\end{array}
\right.
\end{align*}
\end{lem}
\begin{proof}
It is routine to verify the lemma by using Proposition \ref{prop:O}(i).
\end{proof}

\begin{lem}\label{lem:t0t2_O}
Suppose that the $\H$-module $O_d(a,b,c)$ is irreducible. Then the following hold:
\begin{enumerate}
\item There exists a basis $\{w_i\}_{i=0}^d$ for $O_d(a,b,c)$ such that 
\begin{align*}
\left(t_0+t_2+(-1)^i\left(\frac{d}{2}-b-i\right)+\frac{1}{2}\right)
w_i
=\left\{
\begin{array}{ll}
w_{i+1} 
\qquad 
&\hbox{for $i=0,1,\ldots,d-1$},
\\
0
\qquad 
&\hbox{for $i=d$}.
\end{array}
\right.
\end{align*}

\item There exists a basis $\{w_i\}_{i=0}^d$ for $O_d(a,b,c)$ such that 
\begin{align*}
\left(t_0+t_3+(-1)^i\left(\frac{d}{2}-c-i\right)+\frac{1}{2}\right)
w_i
=\left\{
\begin{array}{ll}
w_{i+1} 
\qquad 
&\hbox{for $i=0,1,\ldots,d-1$},
\\
0
\qquad 
&\hbox{for $i=d$}.
\end{array}
\right.
\end{align*}
\end{enumerate}
\end{lem}
\begin{proof}
Let $V$ denote the $\H$-module $O_d(a,b,c)$.

(i): By Definition \ref{defn:H} there exists a unique algebra automorphism $\rho:\H\to \H$ given by 
\begin{eqnarray*}
(t_0,t_1,t_2,t_3) 
&\mapsto &
(t_0,t_2,t_3,t_1)
\end{eqnarray*}
whose inverse sends 
$(t_0,t_1,t_2,t_3)$ to
$(t_0,t_3,t_1,t_2)$. By Lemma \ref{lem:det_O} the traces of $t_0+t_1+\frac{1}{2},t_0+t_2+\frac{1}{2},t_0+t_3+\frac{1}{2}$ act on $V^\rho$ are $b,c,a$ 
respectively.
Therefore the $\H$-module $V^\rho$ is isomorphic to 
$
O_d(b,c,a) 
$
by Theorem \ref{thm:onto2_O}. 
It follows from Lemma \ref{lem:t0t1_O} that there exists a basis $\{w_i\}_{i=0}^d$ for $V^\rho$ such that 
\begin{align*}
\left(t_0+t_1+(-1)^i\left(\frac{d}{2}-b-i\right)+\frac{1}{2}\right)
w_i
=\left\{
\begin{array}{ll}
w_{i+1} 
\qquad 
&\hbox{for $i=0,1,\ldots,d-1$},
\\
0
\qquad 
&\hbox{for $i=d$}.
\end{array}
\right.
\end{align*}
Observe that the action of $t_0+t_1$ on $V^\rho$ is identical to the action of $t_0+t_2$ on $V$. Hence (i) follows.

(ii): By Definition \ref{defn:H} there exists a unique algebra automorphism $\rho:\H\to \H$ given by 
\begin{eqnarray*}
(t_0,t_1,t_2,t_3) 
&\mapsto &
(t_0,t_3,t_2,t_1)
\end{eqnarray*}
whose inverse sends 
$(t_0,t_1,t_2,t_3)$ to
$(t_0,t_3,t_2,t_1)$. By Lemma \ref{lem:det_O} the traces of $t_0+t_1+\frac{1}{2},t_0+t_2+\frac{1}{2},t_0+t_3+\frac{1}{2}$ on $V^\rho$ are $c,b,a$ respectively. Therefore the $\H$-module $V^\rho$ is isomorphic to 
$
O_d(c,b,a)
$
by Theorem \ref{thm:onto2_O}. 
It follows from Lemma \ref{lem:t0t1_O} that there exists a basis $\{w_i\}_{i=0}^d$ for $V^\rho$ such that 
\begin{align*}
\left(t_0+t_1+(-1)^i\left(\frac{d}{2}-c-i\right)+\frac{1}{2}\right)
w_i
=\left\{
\begin{array}{ll}
w_{i+1} 
\qquad 
&\hbox{for $i=0,1,\ldots,d-1$},
\\
0
\qquad 
&\hbox{for $i=d$}.
\end{array}
\right.
\end{align*}
Observe that the action of $t_0+t_1$ on $V^\rho$ is identical to the action of $t_0+t_3$ on $V$. Hence (ii) follows.
\end{proof}

\begin{lem}\label{lem:t1t0_diag_O}
Suppose that the $\H$-module $O_d(a,b,c)$ is irreducible. Then the following are equivalent:
\begin{enumerate}
\item $t_0+t_1$ is diagonalizable on $O_d(a,b,c)$.

\item $t_0+t_1$ is multiplicity-free on $O_d(a,b,c)$.

\item $2a$ is not among $d-1,d-3,\ldots,1-d$.
\end{enumerate}
\end{lem}
\begin{proof} 
Similar to the proof for Lemma \ref{lem:t1t0_diag}.
\end{proof}

\begin{lem}\label{lem:t2t0_diag_O}
Suppose that the $\H$-module $O_d(a,b,c)$ is irreducible. Then the following are equivalent:
\begin{enumerate}
\item $t_0+t_2$ is diagonalizable on $O_d(a,b,c)$.

\item $t_0+t_2$ is multiplicity-free on $O_d(a,b,c)$.

\item $2b$ is not among $d-1,d-3,\ldots,1-d$.
\end{enumerate}
\end{lem}
\begin{proof}
In view of Lemma \ref{lem:t0t2_O}(i) the lemma follows by an argument similar to the proof for Lemma \ref{lem:t1t0_diag_O}.
\end{proof}

\begin{lem}\label{lem:t3t0_diag_O}
Suppose that the $\H$-module $O_d(a,b,c)$ is irreducible. Then the following are equivalent:
\begin{enumerate}
\item $t_0+t_3$ is diagonalizable on $O_d(a,b,c)$.

\item $t_0+t_3$ is multiplicity-free on $O_d(a,b,c)$.

\item $2c$ is not among $d-1,d-3,\ldots,1-d$.
\end{enumerate}
\end{lem}
\begin{proof}
In view of Lemma \ref{lem:t0t2_O}(ii) the lemma follows by an argument similar to the proof for Lemma \ref{lem:t1t0_diag_O}.
\end{proof}

\begin{lem}\label{lem:ABC_diag_O}
Suppose that the $\H$-module $O_d(a,b,c)$ is irreducible. Then the following hold:
\begin{enumerate}
\item If $2a$ is not among $d-1,d-3,\ldots,3-d$ then $A$ is diagonalizable on $O_d(a,b,c)$.

\item If $2b$ is not among $d-1,d-3,\ldots,3-d$ then $B$ is diagonalizable on $O_d(a,b,c)$.

\item If $2c$ is not among $d-1,d-3,\ldots,3-d$ then $C$ is diagonalizable on $O_d(a,b,c)$.
\end{enumerate}
\end{lem}
\begin{proof}
(i): Given any element $X$ of $\H$, let $[X]$ denote the matrix representing $X$ with respect to $\{v_i\}_{i=0}^d$. 
Using Lemma \ref{lem:t0t1_O} a direct calculation yields that $[A]$ is a lower triangular matrix of the form
\begin{gather}\label{[A]_O}
\begin{pmatrix}
\theta_0 & & & & & & &{\bf 0}
\\
 &\theta_1
\\
  & &\theta_1
\\
  &  & &\theta_2
\\
& & & &\theta_2
\\
 &  & & & &\ddots
\\
&  & & & & &\theta_{\frac{d}{2}}
\\  
* & & & & & &0 &\theta_{\frac{d}{2}}
\end{pmatrix}
\end{gather}
where 
$$
\theta_i= \left(\frac{a}{2}-\frac{d+3}{4}+i\right)\left(\frac{a}{2}-\frac{d+3}{4}+i+1\right)
\qquad 
\hbox{for $i=0,1,\ldots,\displaystyle\frac{d}{2}$}.
$$

Since $2a$ is not among $d-1,d-3,\ldots,3-d$, the scalars $\{\theta_i\}_{i=0}^{\frac{d}{2}}$ are mutually distinct. Hence the $\theta_0$-eigenspace of $A$ in $O_d(a,b,c)$ has dimension one and the $\theta_i$-eigenspace of $A$ in $O_d(a,b,c)$ has dimension less than or equal to two for all $i=1,2,\ldots,\frac{d}{2}$. By (\ref{[A]_O}) the last two columns of $[A-\theta_{\frac{d}{2}}]$ are zero. Hence the $\theta_\frac{d}{2}$-eigenspace of $A$ in $O_d(a,b,c)$ has dimension two.
To see the diagonalizability of $A$ it remains to show that the $\theta_i$-eigenspace of $A$ in $O_d(a,b,c)$ has dimension two for all $i=1,2,\ldots,\frac{d}{2}-1$.

By Lemma \ref{lem:t0t1_O} the matrix $[t_0+t_1]$ is a lower triangular matrix of the form 
\begin{gather*}
\begin{pmatrix}
\vartheta_0 & & & & & & &{\bf 0}
\\
 &-\vartheta_1
\\
  & &\vartheta_1
\\
  &  & &-\vartheta_2
  \\
&  &  & &\vartheta_2
\\
&  & & & &\ddots
\\
&  & & & & &-\vartheta_{\frac{d}{2}}
\\  
* & & & & & & &\vartheta_{\frac{d}{2}}
\end{pmatrix}
\end{gather*}
where 
$$
\vartheta_i=a-\frac{d+1}{2}+2i
\qquad 
\hbox{for $i=0,1,\ldots,\displaystyle\frac{d}{2}$}.
$$
Let $i \in\{1,2,..,\frac{d}{2}-1\}$ be given. Let $u$ and $w$ denote the eigenvectors of $t_0+t_1$ in $O_d(a,b,c)$ corresponding to the eigenvalues $\vartheta_i$ and $-\vartheta_{i}$, respectively. Since $2a$ is not among $d-3,d-7,\ldots,5-d$, the scalars $\vartheta_i$ and $-\vartheta_{i}$ are distinct. It follows that $u$ and $w$ are linearly independent. Observe that $u$ and $w$ are the $\theta_i$-eigenvectors of $A$. Hence the $\theta_i$-eigenspace of $A$ in $O_d(a,b,c)$ has dimension. The statement (i) follows.

(ii), (iii): Using Lemma \ref{lem:t0t2_O}(i), (ii) the statements (ii), (iii) follow by the arguments similar to the proof for (i).
\end{proof}

Recall the finite-dimensional irreducible $\Re$-modules from \S\ref{s:AWmodule} and the odd-dimensional irreducible $\H$-modules from \S\ref{Pre:Hmodule2}.

\begin{thm}
[\S 4.6, \cite{Huang:R<BImodules}]
\label{thm:CF_odd}
Suppose that the $\H$-module $O_d(a,b,c)$ is irreducible. Then the following hold:
\begin{enumerate}
\item If $d=0$ then the $\Re$-module $O_d(a,b,c)$ is irreducible and it is isomorphic to 
\begin{align*}
R_0\left(-\frac{a}{2}-\frac{1}{4},-\frac{b}{2}-\frac{1}{4},-\frac{c}{2}-\frac{1}{4}\right).
\end{align*}

\item If $d\geq 2$ and $a+b+c=\frac{d+1}{2}$ then the factors of any composition series for the $\Re$-module $O_d(a,b,c)$ are isomorphic to 
\begin{align*}
&R_{\frac{d}{2}-1}\left(-\frac{a}{2}-\frac{3}{4},-\frac{b}{2}-\frac{3}{4},-\frac{c}{2}-\frac{3}{4}\right),
\\
&R_0\left(-\frac{b+c+1}{2},-\frac{a+c+1}{2},-\frac{a+b+1}{2}\right).
\end{align*}

\item If $d\geq 2$ and $a+b+c\not=\frac{d+1}{2}$ then the factors of any composition series for the $\Re$-module $O_d(a,b,c)$ are isomorphic to
\begin{align*}
&R_{\frac{d}{2}}\left(-\frac{a}{2}-\frac{1}{4},-\frac{b}{2}-\frac{1}{4},-\frac{c}{2}-\frac{1}{4}\right),
\\
&R_{\frac{d}{2}-1}\left(-\frac{a}{2}-\frac{3}{4},-\frac{b}{2}-\frac{3}{4},-\frac{c}{2}-\frac{3}{4}\right).
\end{align*}
\end{enumerate}
\end{thm}

We now give the necessary and sufficient conditions for $A,B,C$ to be multiplicity-free on the composition factors of odd-dimensional irreducible $\H$-modules.

\begin{lem}\label{lem:diag_ABC_O_k02=1}
Suppose that the $\H$-module $O_d(a,b,c)$ is irreducible. If $a+b+c=\frac{d+1}{2}$ then the following hold:
\begin{enumerate}
\item $A$ 
is multiplicity-free on all composition factors of the $\Re$-module $O_d(a,b,c)$ if and only if 
$2a$ is not among $d-5,d-7,\ldots,3-d$.

\item $B$
is multiplicity-free on all composition factors of the $\Re$-module $O_d(a,b,c)$ if and only if 
$2b$ is not among $d-5,d-7,\ldots,3-d$.

\item $C$ 
is multiplicity-free on all composition factors of the $\Re$-module $O_d(a,b,c)$ if and only if 
$2c$ is not among $d-5,d-7,\ldots,3-d$.
\end{enumerate}
\end{lem}
\begin{proof}
Immediate from Lemma \ref{lem:dia_UAW} and Theorem \ref{thm:CF_odd}(i), (ii).
\end{proof}

\begin{lem}\label{lem:diag_ABC_O_k02neq1}
Suppose that the $\H$-module $O_d(a,b,c)$ is irreducible. If $a+b+c\not=\frac{d+1}{2}$ then the following hold:
\begin{enumerate}
\item $A$ 
is multiplicity-free on all composition factors of the $\Re$-module $O_d(a,b,c)$ if and only if 
$2a$ is not among $d-1,d-3,\ldots,3-d$.

\item $B$
is multiplicity-free on all composition factors of the $\Re$-module $O_d(a,b,c)$ if and only if 
$2b$ is not among $d-1,d-3,\ldots,3-d$.

\item $C$ 
is multiplicity-free on all composition factors of the $\Re$-module $O_d(a,b,c)$ if and only if 
$2c$ is not among $d-1,d-3,\ldots,3-d$.
\end{enumerate}
\end{lem}
\begin{proof}
Immediate from Lemma \ref{lem:dia_UAW} and Theorem \ref{thm:CF_odd}(i), (iii).
\end{proof}

We are in the position to prove Theorem \ref{thm:diagonal} in odd-dimensional case.

\begin{thm}\label{thm:diagonal_odd}
Suppose that $V$ is an odd-dimensional irreducible $\H$-module. Then the following are equivalent:
\begin{enumerate}
\item $A$ {\rm (}resp. $B${\rm )} {\rm (}resp. $C${\rm )} is diagonalizable on $V$.

\item $A$ {\rm (}resp. $B${\rm )} {\rm (}resp. $C${\rm )} is diagonalizable on all composition factors of the $\Re$-module $V$.

\item $A$ {\rm (}resp. $B${\rm )} {\rm (}resp. $C${\rm )} is multiplicity-free on all composition factors of the $\Re$-module $V$.
\end{enumerate}
\end{thm}
\begin{proof}
(i) $\Rightarrow$ (ii): Trivial.

(ii) $\Rightarrow$ (iii): Immediate from Lemma \ref{lem:dia_UAW}.

(iii) $\Rightarrow$ (i): 
Suppose that (iii) holds. 
Let $d=\dim V-1$. 
By Theorem \ref{thm:irr_O}(ii) there exist $a,b,c \in \F$ such that $\H$-module $O_d(a,b,c)$ is isomorphic to $V$.

Suppose that $a+b+c=\frac{d+1}{2}$. Since the $\H$-module $O_d(a,b,c)$ is irreducible it follows from Theorem \ref{thm:irr_O}(i) that none of $2a$, $2b$, $2c$ is among $1,3,\ldots,d-1$. Combined with Lemmas \ref{lem:ABC_diag_O} and \ref{lem:diag_ABC_O_k02=1} this yields that (i) holds. Suppose that $a+b+c\not=\frac{d+1}{2}$. By Lemmas \ref{lem:ABC_diag_O} and \ref{lem:diag_ABC_O_k02neq1} the statement (i) holds. Therefore (i) follows.
\end{proof}

\noindent{\it Proof of Theorem \ref{thm:lp}}. It is immediate from Theorem \ref{thm:diagonal} and Proposition \ref{lem:lp}.  \hfill $\square$

\noindent{\it Proof of Theorem \ref{thm:lt}}. It is immediate from Theorem \ref{thm:diagonal} and Proposition \ref{lem:lt}. \hfill $\square$



\bibliographystyle{amsplain}
\bibliography{MP}

\providecommand{\bysame}{\leavevmode\hbox to3em{\hrulefill}\thinspace}
\providecommand{\MR}{\relax\ifhmode\unskip\space\fi MR }
\providecommand{\MRhref}[2]{%
  \href{http://www.ams.org/mathscinet-getitem?mr=#1}{#2}
}
\providecommand{\href}[2]{#2}
\begin{thebibliography}{10}

\bibitem{cur2007}
B.~Curtin, \emph{{Modular Leonard triples}}, Linear Algebra and its
  Applications \textbf{424} (2007), 510--539.

\bibitem{R&BI2015}
V.X. Genest, L.~Vinet, and A.~Zhedanov, \emph{{Embeddings of the Racah algebra
  into the Bannai--Ito algebra}}, SIGMA \textbf{11} (2015), 050, 11 pp.

\bibitem{BI&NW2016}
\bysame, \emph{{The non-symmetric Wilson polynomials are the Bannai--Ito
  polynomials}}, Proceedings of the American Mathematical Society \textbf{144}
  (2016), 5217--5226.

\bibitem{Groenevelt2007}
W.~Groenevelt, \emph{{Fourier transforms related to a root system of rank
  $1$}}, Transformation Groups \textbf{12} (2007), 77--116.

\bibitem{Huang:BImodule}
H.-W. Huang, \emph{{Finite-dimensional irreducible modules of the Bannai--Ito
  algebra at characteristic zero}}, Letters in Mathematical Physics
  \textbf{110} (2020), 2519--2541.

\bibitem{Huang:R<BI}
\bysame, \emph{{The Racah algebra as a subalgebra of the Bannai--Ito algebra}},
  SIGMA \textbf{16} (2020), 075, 15 pages.

\bibitem{Huang:R<BImodules}
\bysame, \emph{{Finite-dimensional modules of the universal Racah algebra and
  the universal additive DAHA of type $(C_1^\vee,C_1)$}}, Journal of Pure and
  Applied Algebra \textbf{225} (2021), 106653.

\bibitem{Huang:DAHA&LT}
\bysame, \emph{{The universal DAHA of type $(C_1^\vee,C_1)$ and Leonard
  triples}}, Communications in Algebra \textbf{49} (2021), 1255--1273.

\bibitem{SH:2019-1}
H.-W. Huang and S.~Bockting-Conrad, \emph{{Finite-dimensional irreducible
  modules of the Racah algebra at characteristic zero}}, SIGMA \textbf{16}
  (2020), 018, 17 pages.

\bibitem{daha&LP}
K.~Nomura and P.~Terwilliger, \emph{{The universal DAHA of type $(C_1^\vee,
  C_1)$ and Leonard pairs of $q$-Racah type}}, Linear Algebra and its
  Applications \textbf{533} (2017), {14--83}.

\bibitem{lp2001}
P.~Terwilliger, \emph{{Two linear transformations each tridiagonal with respect
  to an eigenbasis of the other}}, Linear Algebra and Its Applications
  \textbf{330} (2001), 149--203.

\end{thebibliography}
\end{document}